\documentclass[runningheads]{llncs}

\usepackage[utf8]{inputenc}
\usepackage{amsmath}
\usepackage{array}
\usepackage{amsfonts}
\usepackage{amssymb}
\usepackage{bussproofs}
\usepackage{tikz}
\usetikzlibrary{shapes.misc}
\usepackage{ stmaryrd }
\usepackage{graphicx}
\usepackage{forest}
\usepackage{xcolor}
\usepackage{mathrsfs}
\usepackage{enumitem}
\usepackage{booktabs}
\usepackage{longtable}
\usepackage{bussproofs}
\usepackage{cancel}

\begin{document}

\title{Base-extension Semantics for Modal Logic}



\author{Timo Eckhardt\inst{1}\orcidID{0000-0002-0386-1864} \and 
David J. Pym\inst{1,2}\orcidID{0000-0002-6504-5838}}
\authorrunning{Timo Eckhardt \and David Pym}
%
\institute{University College London, London WC1E 6BT, UK \and
Institute of Philosophy, University of London,  London WC1H 0AR, UK
\email{\{t.eckhardt@ucl.ac.uk,d.pym@ucl.ac.uk\}}}
%


\date{} 


\maketitle

\begin{abstract}

  In proof-theoretic semantics, meaning is based on inference. It may be seen as the mathematical expression of the inferentialist interpretation of logic. Much recent work has focused on base-extension semantics, in which the validity of formulas is given by an inductive definition generated by provability in a `base' of atomic rules. Base-extension semantics for classical and intuitionistic propositional logic have been explored by several authors. In this paper, we develop base-extension semantics for the classical propositional modal systems $K$, $KT$, $K4$, and $S4$, with $\square$ as the primary modal operator. We establish appropriate soundness and completeness theorems and establish the duality between $\square$ and a natural presentation of $\lozenge$. We also show that our semantics is in its current form not complete with respect to euclidean modal logics. Our formulation makes essential use of relational structures on bases.


\end{abstract}

\section{Introduction}

In proof-theoretic semantics (Pt-S), meaning is based on inference. Its philosophical basis in inferentialism \cite{brandom2009} ensures a universal approach to the meaning of the operators of logical systems: their meaning is given in terms of their use. By contrast, in model-theoretic semantics (Mt-S) meaning is determined relative to a choice of abstract (model) structure. 
 
Current research in Pt-S largely follows two different approaches. The first is proof-theoretic validity, following Dummett \cite{Dummett1991} and Prawitz \cite{Prawitz1971,Prawitz2006}. 
It aims to define what makes a proof valid. Dummett developed a philosophical interpretation of the normality results for the natural deduction rules of intuitionistic propositional logic. A valid proof is then one that can, by certain fixed operations, be transformed into a proof without unnecessary detours called a canonical proof. This approach is closely related to the BHK interpretation of IPL. For a more in-depth explanation see \cite{SchroederHeister2006,SchroederHeister2008}.

The second is base-extension semantics (B-eS), as developed in, for example, \cite{Makinson2014,piecha2017definitional,Sandqvist2005,Sandqvist2009,Sandqvist2015,SchroederHeister2016}. In B-eS, sets of atomic rules, called bases, establish the validity of atoms and the meaning of the logical connectives is then inductively defined relative to these bases. The choice of the form of atomic rules has profound repercussions: In the usual set-up of B-eS as stated above, limiting the atomic rules to simple production rules yields semantics for classical logic (see \cite{Sandqvist2005,Sandqvist2009}), while allowing for the discharging of atoms as assumptions, and taking care of the treatment of disjunction, results in intuitionistic semantics \cite{Sandqvist2015}. However, see \cite{Goldfarb2016,Stafford2023} for alternative approaches that are related to Kripke semantics. Deep connections between B-eS and definite formulae, proof-search, logic programming, and negation-as-failure (for intuitionistic propositional logic) have been described in \cite{GP-BSecLog-2023,SchroederHeister2008}.

The standard M-tS for modal logic is Kripke semantics \cite{kripke1959,kripke1963}. In this semantics, the validity of formulas is defined by reference to sets of possible worlds and a relation over those worlds. A necessitated formula (a formula of the form $\square\phi$) is true at a world if it holds at all the worlds reachable via the relation. Different modal logics arise according to different restrictions on the relation.

The goal of this paper is to use the B-eS for classical propositional logic to develop semantics for some of the most common classical propositional modal logics. We also show that our current formulation of this approach does not yield a complete semantics for euclidean modal logics.

The fundamental work on B-eS has been focused on IPL because of the constructive nature of the semantics. There is, however, a relation between intuitionistic and $S4$ modal logic. In fact, there are translations between these two logics, for example \cite{gore2019}. Although comparing a B-eS for intuitionistic propositional logic (for example in \cite{Sandqvist2015}) and $S4$ modal logic, as established later in this paper, is outside the scope of this paper and left for future work (see Section \ref{sec:Conclusion}), this highlights the importance of a B-eS of modal logic as a part of the bigger picture of how proof-theoretic semantics work for different systems. Additionally, B-eS allows us to give a validity relation that is based on proof-theoretic ideas, but closely resembles the modal validity relations reader are most likely already familiar with, given the near universal prevalence of Kripke semantics.

Our strategy is to define relations between bases analogously to the way in which Kripke semantics employs relations between worlds. The meaning of $\square$ is given by reference to relations between bases. This means that the validity of formulas is not solely based on the atomic rules of the bases at which they are evaluated. We do not consider the use of relational structure on bases inappropriate. Even in the case of classical propositional logic, validity already depends on the superset relation and we consider it appropriate to employ also other relations. 

In Section \ref{Sec:Classical}, we give a basic B-eS for classical propositional logic based on \cite{Makinson2014,Sandqvist2005,Sandqvist2009}. This is the underlying set-up. 

After giving a brief introduction to Kripke semantics and giving a Hilbert proof system for modal logic in Section \ref{sec:Kripke}, we develop our B-eS for modal logic in Section \ref{sec:ModalBES}, with $\square$ as our primary modal operator. Additionally, we generalize some important lemmas from the classical to the modal base-extension semantics and establish that at maximally-consistent bases and worlds in Kripke semantics our logical connectives behave in the same way. 

In Section \ref{sec:Completeness}, we proceed to show that the semantics is sound and complete for the modal logic $K$ and, with the appropriate restriction, $KT$, $K4$ and $KT4/S5$. We introduce a natural deduction proof system for classical logic and adapt the proofs in \cite{Sandqvist2005,Sandqvist2009} to show that every formula that is a theorem in this proof system is also valid on our base-extension semantics.\footnote{A proof for the validity of classical tautologies purely between base-extension semantics and the Hilbert system can be found in Appendix \ref{Appendix}.} We take this detour through natural deduction to show that our modal B-eS is an extension of the classical B-eS of \cite{Sandqvist2005,Sandqvist2009} and leaves the classical fragments unchanged. For modal formulas, we then show that the additional modal axioms and the rules of our Hilbert system hold on the corresponding base-extension semantics. This establishes completeness. For soundness we adapt the proof in \cite{Makinson2014} to show that if a formula is not valid in Kripke semantics, it also is not valid in our new semantics. To do so we construct maximally-consistent bases that correspond to the worlds in the model and define an appropriate relation between those bases such that a formula is true at a world if and only if it also holds at the corresponding base. It follows that if a formula can be false in Kripke semantics, there is a base and relation pair at which the formula will not hold. We also give a natural definition of $\lozenge$ that is dual to the $\square$. 

Finally, in Section \ref{sec:Euclidean}, we show that our semantics is not complete for euclidean modal logics.

\section{Classical Base-extension Semantics}
\label{Sec:Classical}
Before addressing the modalities, it is instructive to give the semantics for the underlying 
classical logic. Classical base-extension semantics has been given by Sandqvist in \cite{Sandqvist2005,Sandqvist2009} and Makinson in \cite{Makinson2014}.  
The basis of our semantics does not fundamentally differ from those presentations.

We start with a propositional language of countably many basic sentences (i.e., formulas that do not contain any logical vocabulary). A base is a set of inference rules for basic formulas, called base rules. Given the basic formulas $p$, $q$, and $r$, a base rule might state that $r$ follows from $p$ and $q$. The resulting deducibility relation at a base is then extended with semantic clauses for our logical connectives to a full consequence relation over the language. Noteworthy here are conditional formulas that are treated as hypothetical statements and evaluated with reference to extensions of the base. For example, for $\phi\to\psi$ to hold at a base we require that at every extensions of the base at which $\phi$ holds $\psi$ also holds. Finally, a formula is taken to be valid if it holds at all bases.

\begin{definition}
\label{ClassicalRule}
A \textit{base rule} is a pair $(L_j, p)$ where $L_j  = \{p_1, \dots, p_n \}$ is a finite (possibly empty) set of basic sentences and $p$ is also a basic sentence. Generally, a base rule will be written as $p_1, \dots, p_n \Rightarrow p$. A \textit{base} $\mathscr{B}$ is any countable collection of base rules. We call the set of all bases $\Omega$ and $\overline{\mathscr{B}}$ is the closure of the empty set under the rules in $\mathscr{B}$.

\end{definition}

As bases are the foundation of validity in our semantics, the choice of the class of base rules is important. Allowing for base rules that can have rules as premisses has been shown in \cite{Sandqvist2015} to result in intuitionistic logic while the more restricted form we use give us classical logic. For an in-depth discussion, see \cite{Sandqvist2022}.

\begin{definition} 
\label{def:classicalLanguage}
For atomic propositions $p$, the language for classical logic is generated by the following grammar:
\[
\phi ::= p \mid \bot \mid \phi\to\phi 
\]
\end{definition}

Note that we do not take $\vee$ or $\neg$ as primitive connective, because with the standard validity conditions these cause problems, as we discuss below.

\begin{definition}
\label{satisfaction}
Classical validity is defined as follows:
\[
\begin{array}{l@{\quad}c@{\quad}l}
 \Vdash_\mathscr{B} p    & \mbox{iff} & \mbox{$p$ is in 
            every set of basic sentences closed under $\mathscr{B}$} \\  
        & & \mbox{(i.e., iff $p\in \overline{\mathscr{B}}$)} \\
\Gamma\Vdash_\mathscr{B} \phi & \mbox{iff} & \mbox{$\Vdash_\mathscr{C} \phi$ for all $\mathscr{C}
            \supseteq \mathscr{B}$ s.t. $\Vdash_\mathscr{C} \psi$ for every $\psi \in \Gamma$} \\ 
\Vdash_\mathscr{B} \phi\to\psi & \mbox{iff} & \mbox{$\phi\Vdash_\mathscr{B} \psi$} \\ 
\Vdash_\mathscr{B} \bot & \mbox{iff} & \mbox{$\Vdash_{\mathscr{B}} p$ for every basic sentence $p$.} \\
\end{array} 
\]
A formula $\phi$ is \textit{valid} iff $\Vdash_\mathscr{B} \phi$ for every base $\mathscr{B}$.
A base $\mathscr{B}$ is {\it inconsistent} iff $\Vdash_\mathscr{B} \bot$ and {\it consistent} otherwise.
\end{definition}

This suffices to give a full classical propositional logic. The other connectives can be obtained from $\to$ and $\bot$, mostly in the usual manner; there are, however, two cases that are worth mentioning: $\neg$ and $\vee$.

First, we define the negation symbol $\neg \phi$ as $\phi\to\bot$ 
\cite{Makinson2014}. This is unusual in a classical setting, but a rather standard move in intuitionistic logic. In fact, even the evaluation of $\bot$ is not the usual method for classical logic in most semantics. A more common evaluation of $\bot$ for classical logic would be $\nVdash_\mathscr{B} \bot$, for all $\mathscr{B}$, or $\bot$ is not satisfied at any base. However, we run into problems with such a valuation: For example, in \cite{Makinson2014} it is shown that  \textit{double negation elimination} fails to hold. Let the base $\mathscr{B}$ be the set of all base rules, including the rule $\Rightarrow p$. Obviously $p$ holds at $\mathscr{B}$, but by definition $\bot$ does not and so, since $\mathscr{B}$ is a superset of all bases, there is no base at which $p\to\bot$ holds. It follows that $\Vdash_\emptyset (p\to\bot)\to\bot$, but, since $\nVdash_\emptyset p$, $\nVdash_\emptyset ((p\to \bot)\to \bot)\to p$.


Second, $\vee$ cannot be taken to be primitive with the standard clause or a base-extension style version such as `$\Vdash_\mathscr{B} \phi \vee \psi$ iff for all $\mathscr{C}\supseteq \mathscr{B}$, $\Vdash_\mathscr{C} \phi$ or $\Vdash_\mathscr{C} \psi$', as the \textit{law of excluded middle} fails at the empty base if one takes the treatment of $\bot$ above, because, for any propositional $p$, neither $p$ nor $p\to\bot$ will hold at the empty base and so $\nVdash_\emptyset p\vee(p\to\bot)$. It can, however, be obtained from $\to$ and $\bot$ in the usual manner. 

This suffices to give the classical propositional semantics on which we build the semantics of modal logic. Some lemmas and definitions are needed. The following lemma is taken from \cite{Sandqvist2009}:

\begin{lemma} \label{ClassicalMonotonicity}
If $\Gamma\Vdash_\mathscr{B} \phi$ and $\mathscr{B}\subseteq\mathscr{C}$, then $\Gamma\Vdash_{\mathscr{C}} \phi$.
    
\end{lemma}


Following \cite{Makinson2014}, we introduce the notion of maximally-consistent bases.

\begin{definition}
    \label{MaxConDef}
    A base $\mathscr{B}$ is maximally-consistent iff it is consistent and for every base rule $\delta$, either $\delta \in \mathscr{B}$ or $\mathscr{B}\cup\{\delta\}$ is inconsistent.

\end{definition}

\noindent Maximally-consistent bases are especially interesting because on them $\bot$ and $\to$ behave in the traditional way for a classical logic. This is used in the soundness proof for the classical base-extension semantics in \cite{Makinson2014}, which we adapt to modal base-extension semantics later.

\begin{lemma}
\label{ClassicalBehaviour}
For any maximally-consistent base $\mathscr{B}$, the following hold:
    \begin{itemize}
        \item $\nVdash_\mathscr{B} \bot$
        \item $\Vdash_\mathscr{B} \phi \to \psi$ iff either $\nVdash_\mathscr{B} \phi$ or $\Vdash_\mathscr{B} \psi$.
    \end{itemize}

\end{lemma}

\noindent A proof of Lemma \ref{ClassicalBehaviour} can be found in \cite{Makinson2014}. 


\begin{lemma}
\label{MaxCon}
For every propositional formula $\phi$, if there is a base $\mathscr{B}$ s.t. $\nVdash_\mathscr{B} \phi$, then there is a maximally-consistent base $\mathscr{B}^*\supseteq\mathscr{B}$ with $\nVdash_{\mathscr{B}^*} \phi$.
\end{lemma}

\begin{proof}
    This proof is also taken from \cite{Makinson2014}, but we highlight the construction as it is an important part of our soundness proof (Theorem \ref{Soundness}). 
    First note that $\mathscr{B}$ must be consistent because everything is valid on inconsistent bases. We proceed by induction on the structure of $\phi$.
    For the base case, let $\phi = p$. We construct the base $\mathscr{B}^*$ from $\mathscr{B}$ in the following way. Let $\sigma_1, \sigma_2, \dots$ be an enumeration of all base rules. $\mathscr{B}_0 = \mathscr{B}$, and $\mathscr{B}_{i+1} = \mathscr{B}_i \cup \{\sigma_{i+1}\}$ if $\nVdash_{\mathscr{B}_i \cup \{\sigma_{i+1}\}} p$ or $\mathscr{B}_{i+1}= \mathscr{B}_i$ otherwise and, finally, $\mathscr{B}^* = \bigcup_{i>0} \mathscr{B_i}$. 
    Clearly, $\nVdash_{\mathscr{B}^*} p$ and there is no consistent $\mathscr{C}\supset \mathscr{B}^*$ s.t. $\nVdash_\mathscr{C} p$. As $\nVdash_{\mathscr{B}^*} p$ we also know that $\mathscr{B}^*$ is consistent. What is left to show is that it is maximal. It is easy to see that, rules of the form $p\Rightarrow q$ are in $\mathscr{B}^*$ for every $q$. As we have shown, for any $\mathscr{C}\supset \mathscr{B}^*$, $\Vdash_\mathscr{C} p$, and so $\Vdash_\mathscr{C} q$ for any $q$, and, finally, $\Vdash_\mathscr{C} \bot$.

    For the inductive step, there are two cases to consider. If $\phi = \bot$, note that there is at least some $p$ s.t. $\nVdash_\mathscr{B} p$ and so we can use it to construct a maximally-consistent base as in the base case. For $\phi = \psi\to\chi$, if $\nVdash_\mathscr{B} \psi\to\chi$ there is a base $\mathscr{C}\supseteq \mathscr{B}$ with $\Vdash_\mathscr{C} \psi$, but $\nVdash_\mathscr{C} \chi$. By induction hypothesis we know there is a maximally-consistent base $\mathscr{B}^*\supseteq \mathscr{C}$ with $\nVdash_{\mathscr{B}^*} \chi$ and, by Lemma \ref{ClassicalMonotonicity}, $\Vdash_{\mathscr{B}^*} \psi$. So, $\nVdash_{\mathscr{B}^*} \phi \to \chi$.
    \qed
\end{proof}

One more lemma about maximally-consistent bases is needed. 

\begin{lemma}
    \label{AlmostMaxCon}
    If a base $\mathscr{B}$ has only a single maximally-consistent base $\mathscr{C}$ s.t. $\mathscr{C}\supseteq\mathscr{B}$, then for all $\phi$,
        $\Vdash_\mathscr{B} \phi \text{ iff } \Vdash_\mathscr{C} \phi$.

\end{lemma}

\begin{proof}
    The left-to-right direction follows from Lemma \ref{ClassicalMonotonicity}.

    For the right-to-left direction, we show the contrapositive. If $\nVdash_\mathscr{B} \phi$, then $\nVdash_\mathscr{C} \phi$. We proceed by induction. For the base case, assume $\nVdash_\mathscr{B} p$. 
     By the construction in Lemma \ref{MaxCon}, there is a maximally-consistent $\mathscr{D}\supseteq\mathscr{B}$ s.t. $\nVdash_\mathscr{D} p$ and, since $\mathscr{C}$ is the only maximally-consistent superset base of $\mathscr{B}$, $\mathscr{D}=\mathscr{C}$. 
    For $\phi = \bot$, we know there is for $\nVdash_\mathscr{B}\bot$, there must be some $q$ s.t. $\nVdash_\mathscr{B} q$ and we proceed as in the base case. Finally, for $\phi = \psi\to\chi$, note that in order for $\nVdash_\mathscr{C} \psi\to\chi$ there must be some base $\mathscr{D}\supseteq\mathscr{C}$ with $\Vdash_\mathscr{D} \psi$ and $\nVdash_\mathscr{D} \chi$. Since $\nVdash_\mathscr{D} \chi$, $\mathscr{D}$ cannot be inconsistent. So, $\mathscr{C}=\mathscr{D}$ and $\Vdash_\mathscr{C} \psi$ and $\nVdash_\mathscr{C} \chi$. Finally, $\Vdash_\mathscr{B} \psi$ and $\nVdash_\mathscr{B} \chi$ follow by the induction hypothesis.
    \qed
    \end{proof}

This concludes our presentation of the base-extension semantics for classical propositional logic, as given in \cite{Makinson2014,Sandqvist2005,Sandqvist2009}. We can now proceed to modal logic. 



\section{Kripke Semantics and Hilbert system for Modal Logic}
\label{sec:Kripke}

Before addressing the semantics of modalities, we first need to extend the language with the modal operators $\square$ and $\lozenge$.  As modal logics are a family of logics, we give a general account of the semantics that holds for all normal modal logics whose semantics can be given by Kripke semantics, using $\gamma$ as a placeholder for the name of the modal logic. Later in this section, we discuss explicitly which logics we are considering here.

\begin{definition} 
\label{def:KripkeLanguage}
For atomic propositions $p$, the language  for modal logic $\gamma$ is generated by the following grammar:
\[
    \phi ::= p \mid \bot \mid \neg \phi \mid \phi\to\phi \mid \square \phi \mid \lozenge\phi 
\]
\end{definition}

This is an extension of the language of classical base-extension semantics in two regards: Obviously, the modal operators $\square$ and $\lozenge$ are added, but we also require $\neg$ as a primitive operator, as in Kripke semantics negation is taken as primitive and not defined using $\to$ and $\bot$. In Section~\ref{sec:Duality}, we establish the usual classical duality between $\square$ and $\lozenge$. 


\begin{definition}
    A frame is a pair $F = \langle W, R\rangle$, in which $W$ is a set of possible worlds, $R$ a  binary relation on $W$. A model $M= \langle F,V\rangle$ is a pair of a frame $F$ and a valuation function $V$ giving the set of worlds in which $p$ is true s.t. $V(p) \subseteq W$ for every $p$.
\end{definition}

For each of the different modal logics to be represented, we must restrict the 
relation $R$ accordingly. $K$ is the weakest modal logic for which we can give Kripke semantics, obtained by simply putting no restriction on $R$. Table \ref{BRA} shows the necessary restrictions to get the right semantics for specific modal logics. These restrictions are called frame conditions. The list of axioms in Table \ref{BRA} represents some of the most widely used frame conditions (but, of course, there are many others that are not considered here). In this paper, we only consider modal logics obtained from combining the axioms in Table \ref{BRA}. We do not make any claims about whether the methods of obtaining base-extension semantics for modal logics discussed in this paper can be used to give semantics for modal logics using other modal axioms. 

\begin{table}[t]
\renewcommand\thetable{1}
\begin{center}
   \[
    \begin{array}{|l@{\quad}|l@{\quad}|l|}
\hline
\mbox{Name} & \mbox{Frame condition} & \mbox{Modal axiom}\\
\hline
K & \mbox{None} & \square(\phi \to \psi) \to (\square \phi \to \square\psi)\\
T & \mbox{Reflexivity} & \square \phi \to \phi\\
4 & \mbox{Transitivity} & \square \phi \to \square \square\phi \\
5 & \mbox{Euclidean} & \lozenge \phi \to \square\lozenge\phi \\
\hline
\end{array}
\]
\end{center}
\caption{\label{BRA} Frame conditions and the corresponding modal axioms}
\end{table} 

\begin{definition}
    Let $\gamma$ be a set of modal axioms of Table \ref{BRA}, we call $\gamma$-frames those frames whose relations satisfy the frame conditions for $\gamma$ according to Table \ref{BRA} and $\gamma$-models those models obtained from $\gamma$-frames.
\end{definition}

Formulas are interpreted at worlds, with validity defined at worlds. We give a definition of validity that is used for all modal logics discussed here.

\begin{definition}
    \label{KripkeValidity}
    Let $F = \langle W, R\rangle$ be a $\gamma$-frame, $M=\langle F,V\rangle$ a $\gamma$-model and $w\in W$ a world. That a formula $\phi$ is true at $(M,w)$---denoted $M,w\vDash^\gamma \phi$---as follows:

\[
\begin{array}{l@{\quad}c@{\quad}l}
M,w\vDash^\gamma p   & \mbox{iff} & w\in V(p) \\
M,w\vDash^\gamma \phi\to\psi & \mbox{iff} & \mbox{if $M,w\vDash^\gamma \phi$, then $M,w\vDash^\gamma \psi$} \\ 
M,w\vDash^\gamma \neg \phi & \mbox{iff} & \mbox{not $M,w\vDash^\gamma \phi$} \\ 
M,w\vDash^\gamma \bot & \mbox{iff} & \mbox{never} \\
M,w\vDash^\gamma \square \phi & \mbox{iff} & \mbox{for all $v$ s.t. $Rwv$, $M,v\vDash^\gamma \phi$}\\
M,w\vDash^\gamma \lozenge \phi & \mbox{iff} & \mbox{there is a $v$ s.t. $Rwv$ and $M,v\vDash^\gamma \phi$}\\

\end{array} 
\]

    







\noindent If a formula $\phi$ is true at all worlds in a model $M$, we say $\phi$ is true in $M$. A formula is valid in a model logic $\gamma$ iff it is true in all $\gamma$-models.


\end{definition}

The other propositional connectives are obtained from $\to$ and $\neg$ in the usual manner. Similarly, we do not strictly need both $\lozenge$ and $\square$ as they are duals of each other in the sense that $\lozenge \phi$ iff $\neg\square\neg\phi$ and $\square \phi$ iff $\neg\lozenge\neg\phi$ hold.



This concludes our overview of Kripke semantics.

Hilbert proof systems for modal logic are especially convenient, as we have axioms that correspond to the restriction on $R$ in the Kripke semantics. Because of this, we use the following Hilbert systems.

\begin{definition} \label{HilbertSystem}
    The proof system for the modal logic $K$ is given by the following axioms and rules:
    \begin{itemize}
        \item 1: $\phi\to(\psi\to\phi)$
        \item 2: $(\phi\to(\psi\to\chi))\to((\phi\to\psi)\to(\phi\to\chi))$
        \item 3: $(\neg\phi\to\neg\psi)\to (\psi\to\phi)$
        \item $K$: $\square(\phi\to\psi)\to(\square\phi\to\square\psi)$
        \item $MP$: If $\phi$ and $\phi\to\psi$, then $\psi$ 
        \item $NEC$: If $\phi$ is a theorem, so is $\square\phi$.
    \end{itemize}
\end{definition}
\noindent The axioms 1-3 together with the rule $MP$ constitute a proof system for classical logic. Proof systems for the other modal logics $\gamma$ are obtained by adding the corresponding axioms from Table \ref{BRA}.

This allows us to state soundness and completeness.

\begin{theorem}
\label{KripkeComp}
The proof system for a modal logic $\gamma$ is sound and complete with respect to the validity given by $\gamma$-models (e.g., $KT$ is sound and complete with respect to the validity given by reflexive models).
\end{theorem}




This concludes our quick look at Kripke semantics and Hilbert systems for modal logics. In Section \ref{sec:ModalBES}, we proceed to develop corresponding base-extension semantics. 


\section{Proof-theoretic Semantics for Modal Logic}
\label{sec:ModalBES}

First, we modify our grammar for modal logic to handle negation as $\phi\to\bot$. This, of course, gives us an extension of the language of the classical base-extension semantics given in Definition \ref{def:classicalLanguage}.  Additionally, we treat the dual operator $\lozenge$ as $\neg\square\neg$ in the usual way. This gives us a language that is different from the language used for Kripke semantics in Definition \ref{def:KripkeLanguage} but an extension of the language for classical base-extension semantics defined in Definition \ref{def:classicalLanguage}.

\begin{definition}
    For atomic propositions $p$, the language for a modal logic $\gamma$ is generated by the following grammar:

    \[\phi := p \mid \bot \mid \phi\to\phi \mid \square \phi\]
\end{definition}

The basic idea for base-extension semantics for modal logic is to take the base 
rules and bases from Definition \ref{ClassicalRule} and, as in in Kripke 
semantics, add a relation on the set of bases. We then no longer evaluate 
formulas purely at a base but rather at a base given a relation. To differentiate 
the two, we continue to use $R$ for a relation in a Kripke model and use 
$\mathfrak{R}$ for a relation on the set of bases. 

We can, however, not 
just use any relation on the set of bases. We call the appropriate 
relations \textit{modal relations}. As with Kripke semantics, we 
obtain different modal logics by enforcing different restrictions on 
relations $\mathfrak{R}$. These will be the same restrictions 
as on the relation between worlds in Kripke semantics.

\begin{definition}
\label{modalRelation}
A relation $\mathfrak{R}$ on the set of bases $\Omega$ is called a modal relation iff, for all $\mathscr{B}$,

\begin{enumerate}[label=(\alph*)]
    \item if $\Vdash_{\mathscr{B}} \bot$, then there is a $\mathscr{C}$ s.t. $\mathfrak{R}\mathscr{B}\mathscr{C}$ and $\Vdash_{\mathscr{C}} \bot$ and, for all $\mathscr{D}$, if $\mathfrak{R}\mathscr{B}\mathscr{D}$, then $\Vdash_{\mathscr{D}} \bot$ 

    \item if $\nVdash_{\mathscr{B}} \bot$, then, for all $\mathscr{C}$, s.t. $\mathfrak{R} \mathscr{B}\mathscr{C}$, $\nVdash_{\mathscr{C}} \bot$

    \item for all $\mathscr{C}$, if $\mathscr{B}$ is consistent and $\mathfrak{R}\mathscr{B}\mathscr{C}$, then either $\mathscr{B}$ is maximally-consistent or there is a $\mathscr{D}\supset \mathscr{B}$ s.t. $\mathfrak{R}\mathscr{D}\mathscr{C}$

    \item for all $\mathscr{C}$, if $\mathfrak{R}\mathscr{B}\mathscr{C}$, then for all $\mathscr{D}\subseteq\mathscr{B}$, $\mathfrak{R}\mathscr{D}\mathscr{C}$.

\end{enumerate}

\noindent A modal relation $\mathfrak{R}$ is called a $\gamma$-modal relation iff $\mathfrak{R}$ satisfies the frame conditions corresponding to $\gamma$.

\end{definition}

\begin{figure}[h]
\begin{center}
\begin{tabular}{cc}

\begin{tikzpicture}[-,shorten >=1pt,node distance =1cm, thick] 
			\node[label=below:{$\bot$}]  (A) at (0,1) {$\mathscr{B}$};
			\node[label=below:{$\bot$}] (C) at (2,2) {$\mathscr{C}$};
			\node[label=below:{$\cancel{\bot}$}] (D) at (2,0) {$\mathscr{D}$};
			\path

(A) edge[->] node[above] {at least one\qquad\:\;\;\:\;\;\;\;\:\;\:\:\;\;\;     } (C)
;		
\draw[->] (A) to node {$\xcancel{\:\;}$} (D);

		\end{tikzpicture}

&

\begin{tikzpicture}[-,shorten >=1pt,node distance =2cm, thick] 

\node[label=below:{$\;$}] (Z) at (0,0) {$\;$};
\node[label=below:{$\cancel{\bot}$}] (A) at (2,1) {$\mathscr{B}$};
\node[label=below:{$\bot$}] (B) [right of=A] {$\mathscr{C}$};

\draw[->] (A) to node {$\xcancel{\:\;}$} (B);

\end{tikzpicture}

\end{tabular}
\end{center}
\caption{\label{11(ab)} Illustration of Definition \ref{modalRelation} $(a)$ on the left and $(b)$ on the right}
\end{figure}
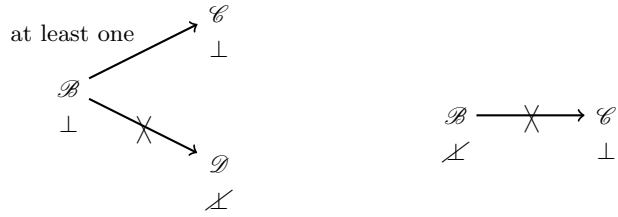

The conditions $(a)$ and $(b)$ ensure that there are no relations between inconsistent and consistent bases, as illustrated in Figure \ref{11(ab)}. Specifically, $(a)$ ensures that we have \textit{ex falso quodlibet}, as we will show in Lemma \ref{EFQ}, and allows us to establish that maximally-consistent bases behave similarly to worlds. Condition $(b)$ establishes that consistent bases do not have relations to inconsistent ones. Again this is used in establishing the similarity between maximally-consistent bases and worlds, because there cannot be inconsistent worlds and so no world will have a relation to it.

Conditions $(c)$ and $(d)$ are best understood when we read $\mathfrak{R}\mathscr{B}\mathscr{C}$ as $\mathscr{C}$ is considered possible at $\mathscr{B}$. This allows us to give intuitive explanations for these conditions. In our proofs,  these two conditions make sure that the structure of the relation $\mathfrak{R}$ is preserved when going up or down the subset relation. This is necessary as the validity of a formula at a base is not necessarily just evaluated at that base, but, as in the case of conditional formulas, can also rely on its supersets. 

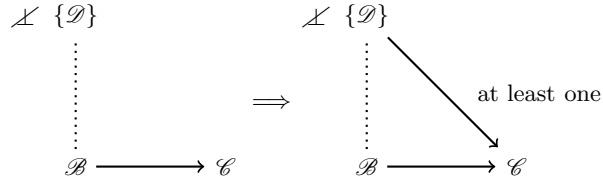
\begin{figure}[h]
\begin{center}

\begin{tabular}{c c  c}

\begin{tikzpicture}[-,shorten >=1pt,node distance =2cm, thick, baseline={([yshift=-1.8ex]current bounding box.center)}] 

\node (A) {$\mathscr{B}$};
\node (B) [right of=A] {$\mathscr{C}$};
\node[label=left:{$\cancel{\bot}$}] (C) [above of=A] {$\{\mathscr{D}\}$};

\draw[->] (A) to node {} (B);
\draw[dotted, right] (A) to node {} (C);

\end{tikzpicture}

&

$\Longrightarrow$

&

\begin{tikzpicture}[-,shorten >=1pt,node distance =2cm, thick, baseline={([yshift=-1.8ex]current bounding box.center)}] 

\node (A) {$\mathscr{B}$};
\node (B) [right of=A] {$\mathscr{C}$};
\node[label=left:{$\cancel{\bot}$}] (C) [above of=A] {$\{\mathscr{D}\}$};

\draw[->] (A) to node {} (B);
\draw[dotted, right] (A) to node {} (C);
\draw[->, right] (C) to node {\quad at least one} (B);

\end{tikzpicture}

\end{tabular}

\caption{\label{11(c)} Illustration of Definition \ref{modalRelation} $(c)$ with dotted lines representing the subset relation;  
$\{\mathscr{D}\}$ denotes that there exists such a $\mathscr{D}$}
\end{center}
\end{figure}

Condition $(c)$ then tells us that if a base $\mathscr{B}$ is not yet maximally-consistent, then for every base $\mathscr{C}$ that is possible at $\mathscr{B}$, there has to be a way of completing $\mathscr{B}$ (i.e., adding base rules to $\mathscr{B}$) such that $\mathscr{C}$ remains possible (see Figure \ref{11(c)}). This is mainly important when looking at those bases that only have a single maximally-consistent base, because this, together with $(d)$, guarantees that they consider the same bases possible and so agree on all modal formulas, as we will see in Lemma \ref{BelowMaxCon}.

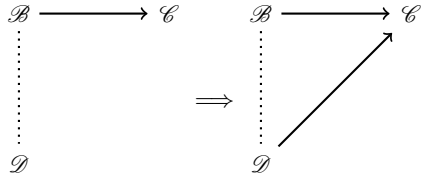
\begin{figure}[h]
\begin{center}

\begin{tabular}{c c  c}

\begin{tikzpicture}[-,shorten >=1pt,node distance =2cm, thick, baseline={([yshift=-1.8ex]current bounding box.center)}] 

\node (A) {$\mathscr{B}$};
\node (B) [right of=A] {$\mathscr{C}$};
\node (C) [below of=A] {$\mathscr{D}$};

\draw[->] (A) to node {} (B);
\draw[dotted, right] (A) to node {} (C);

\end{tikzpicture}

&

$\Longrightarrow$

&

\begin{tikzpicture}[-,shorten >=1pt,node distance =2cm, thick, baseline={([yshift=-1.8ex]current bounding box.center)}] 

\node (A) {$\mathscr{B}$};
\node (B) [right of=A] {$\mathscr{C}$};
\node (C) [below of=A] {$\mathscr{D}$};

\draw[->] (A) to node {} (B);
\draw[dotted, right] (A) to node {} (C);
\draw[->] (C) to node {} (B);

\end{tikzpicture}

\end{tabular}

\caption{\label{11(d)} Illustration of Definition \ref{modalRelation} $(d)$ with dotted lines representing the subset relation}
\end{center}
\end{figure}

Conversely, $(d)$ says that if $\mathscr{C}$ is considered possible (i.e., $\mathfrak{R}\mathscr{B}\mathscr{C}$) at $\mathscr{B}$, then it must also be possible at the subsets of $\mathscr{B}$. So adding new information to a base in the form of base rules cannot make bases possible that are not already possible without that information. New base rules can only cut down on the number of bases considered possible. This conditions will do a lot of the heavy lifting in our proof of completeness in Theorem \ref{Completeness} and the lemmas leading up to it. It guarantees that any conclusion that can be drawn from the bases considered possible at a base can also be drawn at its subsets. In Lemma \ref{OtherModalaxioms}, for example, we make use of this to show that the $(4)$ axiom holds on transitive $\mathfrak{R}$.

Given that we can give the following validity conditions for a modal logic $\gamma$:

\begin{definition}
 \label{EXTValidity}

We define validity at a base $\mathscr{B}$ given a $\gamma$-modal relation $\mathfrak{R}$ for a modal logic $\gamma$ as follows:
\[
\begin{array}{l@{\quad}c@{\quad}l}
\Vdash^\gamma_{\mathscr{B},\mathfrak{R}} p   & \mbox{iff} & \mbox{$p$ is in every set of basic sentences closed under $\mathscr{B}$} \\  
        & & \mbox{(i.e., iff $p\in \overline{\mathscr{B}}$)} \\
\Gamma\Vdash^\gamma_{\mathscr{B},\mathfrak{R}} \phi & \mbox{iff} & \mbox{$\Vdash^\gamma_{\mathscr{C},\mathfrak{R}} \phi$ for all $\mathscr{C}\supseteq \mathscr{B}$ s.t. $\Vdash^\gamma_{\mathscr{C},\mathfrak{R}} \psi$ for every $\psi \in \Gamma$}\\ 
\Vdash^\gamma_{\mathscr{B},\mathfrak{R}} \phi\to\psi & \mbox{iff} & \mbox{$\phi\Vdash^\gamma_{\mathscr{B},\mathfrak{R}} \psi$} \\ 
\Vdash^\gamma_{\mathscr{B},\mathfrak{R}} \bot & \mbox{iff} & \mbox{$\Vdash^\gamma_{\mathscr{B},\mathfrak{R}} p$ for every basic sentence $p$} \\
\Vdash^\gamma_{\mathscr{B},\mathfrak{R}} \square \phi & \mbox{iff} & \mbox{for all $\mathscr{C}\supseteq \mathscr{B}$ and $\mathscr{C'}$ s.t. $\mathfrak{R}\mathscr{C}\mathscr{C'}$, $\Vdash^\gamma_{\mathscr{C'},\mathfrak{R}} \phi$}\\

\end{array} 
\]








\noindent A formula $\phi$ is \textit{$\gamma$-valid}, written as $\Vdash^\gamma \phi$, iff $\Vdash^\gamma_{\mathscr{B}, \mathfrak{R}} \phi$ for all modal bases $\mathscr{B}$ and $\gamma$-modal relations $\mathfrak{R}$. 
Our definition of an inconsistent base from the classical base-extension semantics in Definition \ref{satisfaction} remains unchanged. A base $\mathscr{B}$ is {\it inconsistent} iff $\Vdash_{\mathscr{B}} \bot$.


\end{definition}

Now we can begin adapting lemmas of the classical base-extension semantics from \cite{Makinson2014,Sandqvist2009} to our modal base-extension semantics.

\begin{lemma}
    \label{ModalMonotonicity}

If $\Gamma\Vdash^\gamma_{\mathscr{B},\mathfrak{R}} \phi$ and $\mathscr{B}\subseteq\mathscr{C}$, then $\Gamma\Vdash^\gamma_{\mathscr{C},\mathfrak{R}} \phi$.
    
\end{lemma}

\begin{proof}
    This proof proceeds in the same way as the proof of Lemma \ref{ClassicalMonotonicity}, by induction on $\phi$. We only have the new $\phi=\square\psi$ to consider. By the definition of validity for $\square$ this follows immediately since for all $\mathscr{D}\supseteq\mathscr{C}$ we also have $\mathscr{D}\supseteq\mathscr{B}$.
    \qed
\end{proof}

\begin{lemma}
    \label{EFQ}
    For all $\gamma$, $\gamma$-modal relations $\mathfrak{R}$ and bases $\mathscr{B}$, if $\mathscr{B}$ is inconsistent, then $\Vdash^\gamma_{\mathscr{B},\mathfrak{R}}\phi$ for all $\phi$.
    
\end{lemma}

\begin{proof}
    We prove this by induction on the complexity of $\phi$. For $\phi = p$ and $\phi = \bot$ this follows immediately from the definition of inconsistent bases and the validity conditions in Definition \ref{EXTValidity}. For $\phi= \psi\to\chi$, we have $\Vdash_{\mathscr{B},\mathfrak{R}} \psi\to\chi$ iff $\phi\Vdash_{\mathscr{B},\mathfrak{R}} \psi$. Since all $\mathscr{C}\supseteq\mathscr{B}$ s.t. $\Vdash_{\mathscr{C},\mathfrak{R}} \psi$ will also be inconsistent by Lemma \ref{ModalMonotonicity}, we know that $\Vdash_{\mathscr{B},\mathfrak{R}} \chi$ by the induction hypothesis. For the final case let $\phi = \square\psi$. By Lemma \ref{ModalMonotonicity}, all $\mathscr{C}\supseteq\mathscr{B}$ are also inconsistent and, by $(a)$ of Definition \ref{modalRelation}, so are all $\mathscr{D}$ s.t. $\mathfrak{R}\mathscr{C}\mathscr{D}$ and we have $\Vdash_{\mathscr{D},\mathfrak{R}} \psi$ by the induction hypothesis.
    \qed
\end{proof}



In the proof of soundness for classical base-extension semantics in \cite{Makinson2014}, Makinson makes use of maximally-consistent bases as they behave like classical valuations.\footnote{In this paper, we follow the notation of \cite{Makinson2014} for soundness and completeness. In \cite{Sandqvist2005,Sandqvist2009,Sandqvist2015} these two notions are reversed.} Similarly, we make use of maximally-consistent bases for the same reason for our soundness proof. Put simply, in our soundness proof we show that for every world in a model there is a corresponding maximally-consistent base. For that we show that maximally-consistent bases, together with a modal relation, behave exactly like worlds in Kripke semantics in Lemma \ref{ModalBehaviour}, below.    

Note that we have not changed our notion of a maximally-consistent base from Definition \ref{MaxConDef} as we have not changed the make-up of our bases.


\begin{lemma}
\label{ModalBehaviour}
For any $\gamma$, $\gamma$-modal relation $\mathfrak{R}$ and maximally-consistent base $\mathscr{B}$, the following hold:

\begin{itemize}
    \item $\nVdash^\gamma_{\mathscr{B},\mathfrak{R}} \bot$,
    \item $ \Vdash^\gamma_{\mathscr{B},\mathfrak{R}} \phi \to \psi$ iff $\nVdash^\gamma_{\mathscr{B},\mathfrak{R}} \phi$ or $\Vdash^\gamma_{\mathscr{B},\mathfrak{R}} \psi$, and
    \item $\Vdash^\gamma_{\mathscr{B},\mathfrak{R}} \square \phi $ iff for all $ \mathscr{C}$ s.t. $\mathfrak{R}\mathscr{B}\mathscr{C}, \Vdash^\gamma_{\mathscr{C},\mathfrak{R}} \phi$
\end{itemize}

\end{lemma}

\begin{proof}
    The classical connectives follow from the same strategy as in Lemma \ref{ClassicalBehaviour}. What is left to show is the modal case for $\phi=\square\psi$. Simply note that for all $\mathscr{D}\supset\mathscr{B}$, $\Vdash^\gamma_{\mathscr{D},\mathfrak{R}} \bot$ and so by $(a)$ of Definition \ref{modalRelation}, all $\mathscr{E}$ s,t, $\mathfrak{R}\mathscr{D}\mathscr{E}$ are inconsistent as well and so $\Vdash^\gamma_{\mathscr{E},\mathfrak{R}} \psi$. So the only bases relevant are the bases $\mathscr{C}$ s.t. $\mathfrak{R}\mathscr{B}\mathscr{C}$. 
    \qed
\end{proof}

We can also adapt Lemma \ref{AlmostMaxCon} for the modal case, which will be required for the soundness proof for reflexive $\gamma$.

\begin{lemma}
    \label{BelowMaxCon}
   For all bases $\mathscr{B}$ and $\mathscr{C}$ s.t. $\mathscr{C}$ is the only base that is a maximally-consistent superset of $\mathscr{B}$ and for all $\phi$, 
    \vspace{.2cm}
   
     $\Vdash^\gamma_{\mathscr{B},\mathfrak{R}} \phi$ iff $\Vdash^\gamma_{\mathscr{C}, \mathfrak{R}} \phi$.
\end{lemma}

\begin{proof}
    This proof follows the same strategy as the proof of Lemma \ref{AlmostMaxCon} for classical base-extension semantics.
    The left-to-right direction follows from Lemma \ref{ModalMonotonicity}.
    For the right-to-left direction, note that only the case of $\phi = \square \psi$ is new from Lemma \ref{AlmostMaxCon}. The other cases proceed as before.

     Suppose $\Vdash^\gamma_{\mathscr{C}, \mathfrak{R}} \square\psi$, so that for all $\mathscr{E}$ s.t. $\mathfrak{R}\mathscr{C}\mathscr{E}$ $\Vdash^\gamma_{\mathscr{E}, \mathfrak{R}} \psi$. By $(d)$ of Definition \ref{modalRelation}, it follows that $\mathfrak{R}\mathscr{B}\mathscr{E}$. Since $\mathscr{C}$ is the only maximally-consistent base s.t. $\mathscr{C}\supseteq\mathscr{B}$, it follows from $(c)$ of the same definition that there is no base $\mathscr{F}$ s.t. $\mathfrak{R}\mathscr{B}\mathscr{F}$ but not $\mathfrak{R}\mathscr{C}\mathscr{F}$ and, for all $\mathscr{D}$ s.t $\mathscr{C}\supseteq\mathscr{D}\supseteq\mathscr{B}$ and all bases $\mathscr{G}$, $\mathfrak{R}\mathscr{F}\mathscr{G}$ iff $\mathfrak{R}\mathscr{C}\mathscr{G}$ iff $\mathfrak{R}\mathscr{B}\mathscr{G}$. Since for all those $\mathscr{G}$ we have $\Vdash^\gamma_{\mathscr{G}, \mathfrak{R}}\psi$, we get $\Vdash^\gamma_{\mathscr{B}, \mathfrak{R}} \square\psi$. 
     \qed
\end{proof}

\section{Soundness and Completeness}
\label{sec:Completeness}

With our proof-theoretic semantics established, it remains to show that it is complete and sound with respect to both the proof system and the corresponding Kripke semantics of modal logic. We do so for the modal logics $K$, $KT$, $K4$, and $S4$.  We address problems with euclidean modal logics in Section \ref{sec:Euclidean}.

\begin{theorem}
    \label{Completeness}
    The following are equivalent for $\gamma = K, KT, K4$, or $S4$ and all  $\phi$:
    \begin{enumerate}
        \item  Validity in B-eS: $\phi$ is valid in the base-extension semantics for modal logic $\gamma$,
        \item  Validity in Kripke semantics: $\phi$ is valid in the Kripke semantics for $\gamma$,
        \item  Theorem: $\phi$ is a theorem of the Hilbert system for $\gamma$.
    \end{enumerate}
\end{theorem}

The equivalence between (2. Validity in Kripke semantics) and (3. Theorem) is well established; see, for example, \cite{Blackburn2001}. So we focus on how our base-extension semantics fits into the picture. 
We adapt the proof \cite{Sandqvist2009} to show that being a theorem of the modal logic implies being valid in our base-extension semantics (i.e., Theorem implies Validity in B-eS). Following \cite{Sandqvist2009}, we show that classical tautologies are valid in our base-extension semantics using a natural deduction proof system rather than a Hilbert system.

In \cite{Sandqvist2009}, Sandqvist gives a proof of classicality (corresponding to what we call completeness) between a natural deduction system and classical base-extension semantics. For modal logic, however, we use a Hilbert proof system. This is because the modular nature of the Hilbert system for modal logic allows us to establish the completeness of different modal logics by analysing different modal axioms. Fortunately, we can adapt Sandqvist's proof for the classical fragment of our modal base-extension semantics using the equivalence between natural deduction and Hilbert proof systems for classical logic. We define a natural deduction proof system for classical logic that is equivalent to the classical portion of our Hilbert system.

The reader might wonder about a proof of the classical portion of completeness directly with the Hilbert system. Such a proof can be found in the Appendix. 
We are taking this, technically unnecessary, step through natural deduction to highlight the connection of this work to the existing work on base-extension semantics and specifically to the semantics for classical logics of \cite{Sandqvist2005} and \cite{Sandqvist2009}. Modal relations and the resulting modal base-extension semantics are a conservative extension of the classical semantics and so the relevant proofs still follow in the same way they did before. Our work is very much an addition to the existing classical semantics that does not change the classical fragments of the semantics.

\begin{definition} \label{def:classical-nd}
The natural deduction system for classical propositional logic is given by the following rules:

\begin{center}
\AxiomC{$[\phi]$}
\noLine
\UnaryInfC{$\psi$}
\RightLabel{$(\to I)$}
\UnaryInfC{$\phi\to\psi$}
\DisplayProof
\quad
\AxiomC{$\phi\to \psi$}
\AxiomC{$\phi$}
\RightLabel{$(\to E)$}
\BinaryInfC{$\psi$}
\DisplayProof
\end{center}

\begin{prooftree}
\AxiomC{$[\phi\to\bot]$}
\noLine
\UnaryInfC{$\bot$}
\RightLabel{$(\bot_c)$}
\UnaryInfC{$\phi$}
\end{prooftree}

If we just have the rules $(\to I)$ and $(\to E)$, we have a natural deduction system for minimal propositional logic.
  
\end{definition}

We require the following lemma:

\begin{lemma}
    \label{NDHilbert}
    A formula $\phi$ can be proven from $\Gamma$ using the Hilbert system for classical logic given in Definition \ref{HilbertSystem} iff it is provable from $\Gamma$, without further open assumptions, in the natural deduction proof system for classical logic given in Definition \ref{def:classical-nd}.
\end{lemma}

\noindent For the Hilbert system see \cite{Mendelson1964} and for natural deduction \cite{Prawitz1965}.

Having set up the necessary background, we proceed by showing that the rules of the natural deduction system for minimal logic and the structural conditions of reflexivity of implication (R), thinning (weakening),  contraction, interchange (collectively, S), and cut (C) (see \cite{Gentzen1935} for these) hold for our base-extension semantics.

\begin{lemma}
\label{MinimalLogic}
        The following hold for all modal logics $\gamma$, $\gamma$-modal relations $\mathfrak{R}$, and bases $\mathscr{B}$:

    \begin{itemize}
        \item (R) $\phi\Vdash^\gamma_{\mathscr{B},\mathfrak{R}} \phi$
        \item (S) if $\Gamma\Vdash^\gamma_{\mathscr{B},\mathfrak{R}} \phi$ and $\Gamma\subseteq\Delta$, then $\Delta\Vdash^\gamma_{\mathscr{B},\mathfrak{R}}\phi$
        \item (C) if $\Gamma\Vdash^\gamma_{\mathscr{B},\mathfrak{R}} \phi$ and $\Gamma, \phi \Vdash^\gamma_{\mathscr{B},\mathfrak{R}}\psi$, then $\Gamma\Vdash^\gamma_{\mathscr{B},\mathfrak{R}}\psi$
        \item ($\to$E) $\phi\to\psi, \phi \Vdash^\gamma_{\mathscr{B},\mathfrak{R}} \psi$
        \item ($\to$I) if $\Gamma,\phi\Vdash^\gamma_{\mathscr{B},\mathfrak{R}}\psi$, then $\Gamma\Vdash^\gamma_{\mathscr{B},\mathfrak{R}}\phi\to\psi$
    \end{itemize}
\end{lemma}

\begin{proof}
    Using the methods of the proof for classical base-extension semantics in \cite{Sandqvist2009}, these follow from Definition \ref{EXTValidity} and Lemma \ref{ModalMonotonicity}. There is nothing in these proofs that depends on the internal structures of the formulas $\phi$ or $\psi$ and so these will still hold when they are modal formulas rather than just propositional ones.
    \qed
\end{proof}

This tells us that the theorems of minimal logic are valid in our semantics. To recover classical logic, it remains to show that the base-extension semantics respects ($\bot_c$).

\begin{lemma}
    \label{DoubleNegation}
            For all $\gamma$, $\gamma$-modal relations $\mathfrak{R}$, and bases $\mathscr{B}$:

\vspace{.2cm}

    ($\bot_c$) $(\phi\to\bot)\to\bot\Vdash^\gamma_{\mathscr{B},\mathfrak{R}}\phi$
\end{lemma}

\begin{proof}
    Similarly to \cite{Sandqvist2009}, we do not show this for arbitrary $\phi$. We follow the strategy of \cite{Sandqvist2009} to reduce the number of instances of ($\bot_c$) we need to consider. In \cite{Prawitz1971}, Prawitz shows that, given the rules of minimal logic, instances of ($\bot_c$) ranging over implication can be reduced to applications of (DN) over its component parts. So, given Lemma \ref{MinimalLogic}, we can reduce all instances of $(\bot_c)$ to either atomic sentences, boxed formulas, or falsum. So, we need only deal with the case in which $\phi = p$, $\square\psi$, or $\bot$.
    
    The proof of $(p\to\bot)\to\bot\Vdash_{\mathscr{B}} p$ of \cite{Sandqvist2009} can be straightforwardly generalized to $(p\to\bot)\to\bot\Vdash^\gamma_{\mathscr{B},\mathfrak{R}} p$, since for any $\mathscr{C}$, $\Vdash^\gamma_{\mathscr{C},\mathfrak{R}} p$ iff $\Vdash_{\mathscr{C}} p$.

    Next we show $(\square\psi\to\bot)\to\bot\Vdash^\gamma_{\mathscr{B},\mathfrak{R}} \square\psi$. That is, by Definition \ref{EXTValidity}, for all $\mathscr{C}\supseteq\mathscr{B}$ if $\Vdash^\gamma_{\mathscr{C},\mathfrak{R}} (\square\psi\to\bot)\to\bot$, then $\Vdash^\gamma_{\mathscr{C},\mathfrak{R}}\square\psi$. We can break $\Vdash^\gamma_{\mathscr{C},\mathfrak{R}} (\square\psi\to\bot)\to\bot$ down further to, for all $\mathscr{D}\supseteq\mathscr{C}$, either $\Vdash^\gamma_{\mathscr{D},\mathfrak{R}} \bot$ or $\nVdash^\gamma_{\mathscr{D},\mathfrak{R}} \square\psi\to\bot$. So we have two cases to consider:
    
    For $\Vdash^\gamma_{\mathscr{D},\mathfrak{R}} \bot$, it follows from Lemma \ref{EFQ} that $\Vdash^\gamma_{\mathscr{D},\mathfrak{R}} \square\psi$ and so, $\Vdash^\gamma_{\mathscr{B},\mathfrak{R}} \square\psi$. 
    
    If $\nVdash^\gamma_{\mathscr{D},\mathfrak{R}} \square\psi\to\bot$, then there exists an $\mathscr{E}\supseteq\mathscr{D}$ s.t. $\Vdash^\gamma_{\mathscr{E},\mathfrak{R}} \square\psi$ and $\nVdash^\gamma_{\mathscr{E},\mathfrak{R}} \bot$, so for all consistent $\mathscr{F}\supseteq \mathscr{E}$ and $\mathscr{G}$ s.t. $\mathfrak{R}\mathscr{F}\mathscr{G}$, $\Vdash^\gamma_{\mathscr{G},\mathfrak{R}} \psi$. By Definition \ref{modalRelation} $(d)$, we know that $\mathfrak{R}\mathscr{B}\mathscr{G}$ and since this holds for all consistent $\mathscr{D}\supset\mathscr{B}$, it follows from $(c)$ of Definition \ref{modalRelation} that for all $\mathscr{G}$ s.t. $
    \mathfrak{R}\mathscr{B}\mathscr{G}$, $\Vdash^\gamma_{\mathscr{G},\mathfrak{R}} \psi$ and so, $\Vdash^\gamma_{\mathscr{B},\mathfrak{R}} \square\psi$.

     For $\phi=\bot$, note that for $\Vdash^\gamma_{\mathscr{C},\mathfrak{R}}(\bot\to\bot)\to\bot$, by Definition \ref{EXTValidity}, $\Vdash^\gamma_{\mathscr{C},\mathfrak{R}}\bot$ or $\nVdash^\gamma_{\mathscr{C},\mathfrak{R}}\bot\to\bot$. In the first case we have shown what we need and the second is impossible as $\bot\to\bot$ holds at all bases.
    \qed
\end{proof}

Lemmas \ref{MinimalLogic} and \ref{DoubleNegation} show us that every classical tautology is also valid in our modal base-extension semantics. Finally, consider  modal formulas. Every modal validity can be obtained from the classical tautologies and/or the modal axioms using the rules modus ponens and necessitation. So we just need to show that, in addition to the classical tautologies, the modal axioms and rules hold in our base-extension semantics. Here, it is important to take extra care about $\gamma$, as the choice of modal logic changes which axioms hold. The reasoning in these proofs is very close to similar reasoning for Kripke semantics. We start with the basic modal logic $K$.

\begin{lemma}
    \label{Kmodalaxioms}
    For any $\gamma$, the following hold:
    \begin{itemize}
        \item (MP) If $\Vdash^\gamma_{\mathscr{B},\mathfrak{R}} \phi\to\psi$ and $\Vdash^\gamma_{\mathscr{B},\mathfrak{R}} \phi$, then $\Vdash^\gamma_{\mathscr{B},\mathfrak{R}} \psi$ 
        \item (NEC) If $\Vdash^\gamma \phi$, then $\Vdash^\gamma \square\phi$
        \item (K) $\square(\phi\to\psi)\Vdash^\gamma\square\phi\to\square\psi$.
    \end{itemize}
\end{lemma}

\begin{proof}
    To show that (MP) holds, note that because $\Vdash^\gamma_{\mathscr{B},\mathfrak{R}} \phi\to\psi$, for all $\mathscr{C}\supseteq\mathscr{B}$, if $\Vdash^\gamma_{\mathscr{C},\mathfrak{R}}\phi$, then $\Vdash^\gamma_{\mathscr{C},\mathfrak{R}}\psi$. By assumption, $\Vdash^\gamma_{\mathscr{B},\mathfrak{R}} \phi$ and 
    $\mathscr{B}\supseteq\mathscr{B}$, and so $\Vdash^\gamma_{\mathscr{B},\mathfrak{R}}\psi$. 

    To prove (NEC), we take the contrapositive. If $\nVdash^\gamma \square\phi$, then $\nVdash^\gamma \phi$ and suppose $\nVdash^\gamma\square\phi$. So, there is a $\mathscr{B}$ and an $\mathfrak{R}$ s.t. $\nVdash^\gamma_{\mathscr{B},\mathfrak{R}} \square\phi$. Therefore, there are $\mathscr{C}\supseteq\mathscr{B}$ and $\mathscr{D}$ s.t. $\mathfrak{R}\mathscr{C}\mathscr{D}$ and $\nVdash^\gamma_{\mathscr{D}, \mathfrak{R}} \phi$ and $\nVdash^\gamma \phi$.  

    By Definition \ref{EXTValidity}, it follows that (K) iff, for all $\mathscr{B}$ and $\mathfrak{R}$, if $\Vdash^\gamma_{\mathscr{B},\mathfrak{R}} \square(\phi\to\psi)$, then $\Vdash^\gamma_{\mathscr{B},\mathfrak{R}} \square\phi\to\square\psi$. Again, we show the contrapositive: for all $\mathscr{B}$ and $\mathfrak{R}$, if $\nVdash^\gamma_{\mathscr{B},\mathfrak{R}} \square\phi\to\square\psi$, then $\nVdash^\gamma_{\mathscr{B},\mathfrak{R}} \square(\phi\to\psi)$. Without loss of generality, take an arbitrary $\mathscr{B}$ such that $\nVdash^\gamma_{\mathscr{B},\mathfrak{R}} \square\phi\to\square\psi$. Then, there is a $\mathscr{C}\supseteq\mathscr{B}$ s.t. $\Vdash^\gamma_{\mathscr{C},\mathfrak{R}} \square\phi$ and $\nVdash^\gamma_{\mathscr{C},\mathfrak{R}} \square\psi$. So, there are $\mathscr{D}\supseteq\mathscr{C}$ and $\mathscr{E}$ s.t. $\mathfrak{R}\mathscr{D}\mathscr{E}$ with $\Vdash^\gamma_{\mathscr{E},\mathfrak{R}} \phi$ but $\nVdash^\gamma_{\mathscr{E},\mathfrak{R}} \psi$ and so $\nVdash^\gamma_{\mathscr{E},\mathfrak{R}} \phi\to\psi$ and $\nVdash^\gamma_{\mathscr{D},\mathfrak{R}} \square(\phi\to\psi)$. Finally, as $\mathscr{D}\supseteq\mathscr{B}$, $\nVdash^\gamma_{\mathscr{B},\mathfrak{R}} \square(\phi\to\psi)$.
          \qed
\end{proof}

This suffices for the modal logic $K$. What remains is to show that the appropriate modal axioms hold for the corresponding restrictions on $\mathfrak{R}$.

\begin{lemma}
    \label{OtherModalaxioms}
    For any $\gamma$, the following hold:
    \begin{itemize}
        \item If $\gamma$ includes reflexivity, then (T)  $\square\phi\Vdash^\gamma \phi$
        \item If $\gamma$ includes transitivity, then (4) $\square \phi\Vdash^\gamma \square\square\phi$.  
    \end{itemize}
    
\end{lemma}

\begin{proof}
    For (T) assume, without loss of generality, some $\mathscr{B}$ and $\mathfrak{R}$ s.t. $\Vdash^\gamma_{\mathscr{B}, \mathfrak{R}} \square \phi$. By reflexivity, we know that $\mathfrak{R}\mathscr{B}\mathscr{B}$ and so $\Vdash^\gamma_{\mathscr{B}, \mathfrak{R}} \phi$.

    For (4), we again, without loss of generality, assume some $\mathscr{B}$ and $\mathfrak{R}$ s.t. $\Vdash^\gamma_{\mathscr{B}, \mathfrak{R}} \square \phi$. For contradiction, we also assume $\nVdash^\gamma_{\mathscr{B}, \mathfrak{R}} \square \square\phi$. So, there are $\mathscr{C}\supseteq\mathscr{B}$ and $\mathscr{D}$ s.t. $\mathfrak{R}\mathscr{C}\mathscr{D}$ and $\nVdash^\gamma_{\mathscr{D}, \mathfrak{R}} \square \phi$. This means there have to be $\mathscr{E}\supseteq\mathscr{D}$ and $\mathscr{F}$ s.t. $\mathfrak{R}\mathscr{E}\mathscr{F}$ and $\nVdash^\gamma_{\mathscr{F}, \mathfrak{R}} \phi$. By Definition \ref{modalRelation} $(d)$, we know $\mathfrak{R} \mathscr{D}\mathscr{F}$ and, by transitivity, $\mathfrak{R}\mathscr{C}\mathscr{F}$ and so $\nVdash^\gamma_{\mathscr{B}, \mathfrak{R}} \square \phi$, which contradicts our assumption.
    \qed
\end{proof}

From Lemmas \ref{NDHilbert}, \ref{MinimalLogic}, \ref{DoubleNegation}, \ref{Kmodalaxioms}, and \ref{OtherModalaxioms}, we get the following:

\begin{theorem}[Completeness] For $\gamma = K, KT, K4$, or $S4$, if $\phi$ is a theorem of modal logic $\gamma$, then $\phi$ is valid in the base-extension semantics for $\gamma$.
\end{theorem}


This establishes the step from (3. Theorem) to (1. Validity in B-eS) for our Theorem \ref{Completeness}. For the last part, we want to show the step from (1. Validity in B-eS) to (2. Validity in Kripke semantics); that is, being valid in a base-extension semantics implies being valid in the corresponding Kripke semantics. This will be our proof of soundness. For this, we follow the proof for classical base-extension semantics from \cite{Makinson2014}.

This soundness proof goes via the contrapositive. We show that if $\phi$ is not valid in the Kripke semantics for $\gamma$, then it also is not valid for the corresponding base-extension semantics. We do so by taking a frame $F = \langle W,R\rangle$ and a model $M = \langle F,V\rangle$ with a world $w$ s.t. $M,w\nvDash^\gamma \phi$ and constructing bases for each of the worlds in $M$ so that the base and the world agree on all validity judgements. Specifically, they will agree on $\phi$ and so $\phi$ will not be valid in base-extension semantics either.

The construction of such a model, and establishment of the required property, is complicated, with many steps whose role is not always easy to see. We give a brief sketch of the proof before moving to the actual proof.


As $\phi$ is not valid in Kripke semantics, we know there is a model and a world s.t. $M,w\nvDash^\gamma \phi$. We cannot simply construct a maximally-consistent base for each world in $M$, because $M$ can have worlds with the same valuation, but it is not possible to have two different maximally-consistent bases with the same base rules. So our first step is to construct a model $M'$ from $M$ where there are no two worlds in $M'$ that have the same valuation, but we still have $M',w\nvDash^\gamma \phi$. Next we use the methods of \cite{Makinson2014} to construct maximally-consistent bases (like $\mathscr{B}_w$) from the worlds in $M'$ (like $w$) so that all non-modal formulas will hold at $\mathscr{B}_w$ iff they are true at the world $w$, by choosing the appropriate base rules. We then define a relation $\mathfrak{R}$ on the set of bases, starting with $\mathfrak{R}\mathscr{B}_w\mathscr{B}_v$ iff $Rwv$, to deal with modal formulas. We then show that $\mathfrak{R}$ is the right type of relation; that is to say a $\gamma$-modal relation, and, finally,  that formulas hold at $\mathscr{B}_w$ iff they are true at $w$. Given that we started with $M,w\nvDash^\gamma \phi$, if follows that $\nVdash^\gamma_{\mathscr{B}_w, \mathfrak{R}} \phi$ and so $\phi$ is not valid in base-extension semantics.

Here is a list of the steps for reference: 

\begin{enumerate}
    \item Construct $M'$ from $M$ s.t. there are no two worlds with the same valuation.
    \item Construct maximally-consistent bases corresponding to the worlds in $M'$.
    \item Define a relation $\mathfrak{R}$ on the set of bases using $R'$ of $M'$
    \item Prove that $\mathfrak{R}$ is a $\gamma$-modal relation
    \item Prove that $\Vdash^\gamma_{\mathscr{B}_w, \mathfrak{R}} \psi$ iff $M,w\vDash^\gamma \psi$, for all $\psi$
    \item Since $M,w\nvDash^\gamma \phi$, it follows that $\nVdash^\gamma_{\mathscr{B}_w, \mathfrak{R}} \phi$. So, $\phi$ is not valid in the base-extension semantics for $\gamma$.
\end{enumerate}

Given this blueprint we now move to the actual proof.

\begin{theorem}[Soundness]
\label{Soundness}
For $\gamma = K, KT, K4$, or $S4$, if $\phi$ is valid in modal base-extension semantics for $\gamma$, it is also valid in the Kripke semantics for $\gamma$.
\end{theorem}

\begin{proof}


As we want to show the contrapositive, we know there is a model $M$ and a world $W$ s.t. $M,w\nVdash^\gamma \phi$. We begin the proof with step 1 from above.

Unlike in \cite{Makinson2014}, we can have different worlds with the same valuation and so we need bases that disagree only on propositional formulas that are not relevant to $\phi$. We achieve that by choosing a different propositional sentence $q$ for each world that does not appear in $\phi$. Note that this is possible because there are infinitely many propositional sentences that can be used and only finitely many of them can be part of $\phi$. So, for every world $v \in W$, let $q_v$ be some propositional sentence s.t. it does not appear in $\phi$ and, for all worlds $u$, if $u \neq v$, then $q_u \neq q_v$. We now construct a model $M'$ by changing the valuations concerning the $q$s to get a model at which no two worlds have the same valuation. We define $M'=\langle F,V'\rangle$ s.t. for all $v\in W$, $V'(q_v) = W\setminus\{v\}$ and $V'(p)=V(p)$ for all other propositional atoms $p$. Since the $q$s do not appear in $\phi$, we clearly have $M',w\nvDash^\gamma \phi$.

For step 2, we now construct the bases and modal relation that correspond to the model $M'$. We first construct maximally-consistent bases that correspond to the classical judgements of the worlds in $W$ and then add the relation onto that (in step 3.). 

We start by defining a base $\mathscr{A}_w$ for every $w\in W$ as follows: 
\[
\begin{split}
\mathscr{A}_w := \; & \{\Rightarrow p : M',w\vDash p\} \cup \\
                 & \{\Rightarrow q_v : v \in W \text{ and } v\neq w\} \cup \\
                 & \{p\Rightarrow q_w, q_w\Rightarrow p : M',w\nvDash p\} \\
\end{split}
\]
 
We now want $\mathscr{B}_w$ to be a maximally-consistent extension of $\mathscr{A_w}$ for which $\nVdash_{\mathscr{C}} q_w$. By Lemma \ref{MaxCon}, we know that there is a maximally-consistent base $\mathscr{C}$ s.t. $\mathscr{C}\supseteq\mathscr{A}_w$ and $\nVdash_{\mathscr{C}} q_w$, so we simply take that $\mathscr{C}$ as our $\mathscr{B}_w$.

For step 3, we now need to define a relation $\mathfrak{R}$ s.t., if $Rwv$, then $\mathfrak{R}\mathscr{B}_w\mathscr{B}_v$. We require $\mathfrak{R}$ to be a $\gamma$-modal relation, however. To this end we consider the relation such that $\mathfrak{R}\mathscr{B}\mathscr{C}$ only by the following:



\begin{enumerate}[label=(\arabic{enumi})]
    \item if $Rwv$, then $\mathfrak{R}\mathscr{B}_w\mathscr{B}_v$
    \item if $\mathscr{B}$ and $\mathscr{C}$ are inconsistent, then $\mathfrak{R}\mathscr{B}\mathscr{C}$
    \item if there is a consistent base $\mathscr{D}$ s.t. $\mathscr{B}\subseteq\mathscr{D}$ and $\mathfrak{R}\mathscr{D}\mathscr{C}$, then $\mathfrak{R}\mathscr{B}\mathscr{C}$

    \item if $\gamma$ includes transitivity, then if $\mathfrak{R}\mathscr{B}\mathscr{C}$ and $\mathfrak{R}\mathscr{C}\mathscr{D}$, $\mathfrak{R}\mathscr{B}\mathscr{D}$
    
    \item if $\gamma$ includes reflexivity, then (5a) for all bases $\mathscr{B}$, $\mathfrak{R}\mathscr{B}\mathscr{B}$, (5b)  
    if $\mathscr{B}$ is maximally-consistent, $\mathscr{B}\supseteq\mathscr{C}$ and there is no $w\in W'$ s.t. $\mathscr{B} = \mathscr{B}_w$, then $\mathfrak{R}\mathscr{B}\mathscr{C}$, 
    and (5c) if $\mathscr{B}_w$ is the only maximally-consistent base s.t. $\mathscr{B}_w\supseteq\mathscr{C}$, $\mathfrak{R}\mathscr{B}_w\mathscr{C}$. 
\end{enumerate}

In the case of reflexivity, it is not enough just to require $\mathfrak{R}$ to be reflexive. As we will see below, (5b) and (5c) are necessary to guarantee that the relation still obeys condition $(c)$ of Definition \ref{modalRelation}.

We have now constructed the bases and the relation we need. What is left to show is that what we have constructed them appropriately. For that we start with step 4 and show that $\mathfrak{R}$ is a $\gamma$-modal relation. For this we first have to show that $\mathfrak{R}$ is a modal relation.



So we have to show that the conditions $(a)$ to $(d)$ of Definition \ref{modalRelation} hold. Note that (2) simply connects all inconsistent bases with each other. We can do that because we do not discuss any frame conditions that exclude relations (like irreflexivity) here. \footnote{Nevertheless, these cases can in principle also be dealt with by a more careful construction of $\mathfrak{R}$.} This guarantees condition $(a)$. Since the only other steps that can connect a base to an inconsistent base are (4) and (5a), and that can only happen if the initial base is itself inconsistent, we also have $(b)$. Condition $(d)$ follows from (3) immediately. 


For $(c)$, we go step-by-step and start with the case in which $\gamma$ does not include transitivity or reflexivity. For any consistent bases $\mathscr{B}$ and $\mathscr{C}$, we can  only have $\mathfrak{R}\mathscr{B}\mathscr{C}$ because of (1) or (3). For (1) we know that $\mathscr{B}$ is maximally-consistent bases and so $(c)$ will trivially hold. For (3) simply note that the antecedent guarantees $(c)$.

Let $\gamma$ be transitive. We show that the relations added by (4) still fulfill condition $(c)$. Suppose $\mathfrak{R}\mathscr{B}\mathscr{C}$ and $\mathfrak{R}\mathscr{C}\mathscr{D}$ and that for these relations condition $(c)$ holds; that is, either $\mathscr{B}$ is maximally-consistent or there is a $\mathscr{E}$ s.t. $\mathfrak{R}\mathscr{E}\mathscr{C}$ and the same for $\mathfrak{R}\mathscr{C}\mathscr{D}$. We show that (4) also holds for $\mathfrak{R}\mathscr{B}\mathscr{D}$. Since $(c)$ holds for $\mathfrak{R}\mathscr{B}\mathscr{C}$, we know that either $\mathscr{B}$ is maximally-consistent, and we are done, or there is a $\mathscr{E}$ s.t. $\mathscr{E}\supset\mathscr{B}$ and $\mathfrak{R}\mathscr{E}\mathscr{C}$. From (4) and $\mathfrak{R}\mathscr{C}\mathscr{D}$, it follows that $\mathfrak{R}\mathscr{E}\mathscr{D}$.


As mentioned above, the case for reflexivity is not as straightforward. By (5a), we add $\mathfrak{R}\mathscr{B}\mathscr{B}$, for all $\mathscr{B}$. Either of (5b) and (5c) guarantees condition $(c)$, when they apply. So, we just need to show that for every consistent base one of them does. So, we need to show that, for all consistent bases $\mathscr{B}$, there is a maximally-consistent $\mathscr{C}\supseteq \mathscr{B}$ s.t. either $\mathscr{C}$ is the only such base and there is a $w\in W$ s.t. $\mathscr{C}=\mathscr{B_w}$ and we have $(c)$ by (5b) or there is no such $w$ and $(c)$ follows from (5c).

It is easy to see that a maximally-consistent $\mathscr{C}\supseteq\mathscr{B}$ exists for all consistent $\mathscr{B}$ and, if there is no $w\in W$ s.t. $\mathscr{C} = \mathscr{B}_w$, we are done. What remains to show is that, if for all maximally-consistent bases $\mathscr{C}\supseteq\mathscr{B}$ there exists a corresponding world $w$ s.t.  $\mathscr{C}=\mathscr{B_w}$, then there is only one such $\mathscr{C}$. Recall that for all $w$ and $v \in W$, $\mathscr{B}_w$ and $\mathscr{B}_v$ disagree on at least two propositional sentences: $q_w$ and $q_v$ and if $\mathscr{B}_w$ and $\mathscr{B}_v$ are supersets of $\mathscr{B}$, then at least one more distinct maximally-consistent bases that is a superset of $\mathscr{B}$: a maximally-consistent base on which neither $q_w$ nor $q_v$ hold. Although the base at which both hold need not still be possible in the presence of rules of the form $q_w,q_v \Rightarrow p$ for all $p$, we can not similarly exclude the case that neither holds. Let $\mathscr{D}$ be that base. We know that $\mathscr{D}\supseteq\mathscr{B}$ and that there is no $\mathscr{B}_u$ s.t. $\mathscr{B}_u = \mathscr{D}$ as $q_u$ holds at $\mathscr{D}$ (because it does at $\mathscr{B}_w$ and $\mathscr{B}_v$). This is a contradiction and we conclude that condition $(c)$ holds if $\gamma$ is reflexive and so we have shown that $\mathfrak{R}$ is a model relation.


The proof of that are is specifically a $\gamma$-modal relation is now straightforward. We have already shown that $\mathfrak{R}$ is a modal relation. So, it follows from (4), for reflexive $\gamma$, and (5a), for transitive $\gamma$. 

In step 5, we now show that the a formula is true at a world $w$ in $M'$ iff it holds at the corresponding base $\mathscr{B}_w$ given $\mathfrak{R}$. We prove this by induction on the complexity of $\phi$. The base case of $\phi = p$ follows directly from construction of the base $\mathscr{B}_w$. 

For $\phi = \bot$ note that $M',w\nvDash^\gamma \bot$ for all $w \in W'$ and that $\mathscr{B}_w$ is consistent. 

For $\phi = \psi\to\chi$, we know, by Lemma \ref{ModalBehaviour}, that $\Vdash^\gamma_{\mathscr{B}_w,\mathfrak{R}} \psi\to\chi$ iff $\nVdash^\gamma_{\mathscr{B}_w,\mathfrak{R}} \psi$ or $\Vdash^\gamma_{\mathscr{B}_w,\mathfrak{R}} \chi$. Similarly, $M',w\vDash^\gamma \psi\to\chi$ iff $M',w\nvDash^\gamma \psi$ or $M',w\vDash^\gamma \chi$  and we conclude by the induction hypothesis. 

We give the case in which $\phi = \square \psi$. Again, reflexivity will complicate this, so we start with the case in which $\gamma$ is not reflexive (but may be transitive). By Lemma \ref{ModalBehaviour}, 
$\Vdash^\gamma_{\mathscr{B}_w,\mathfrak{R}} \square\psi$ iff for all $\mathscr{C}$ s.t. $\mathfrak{R}\mathscr{B}_w\mathscr{C}$ $\Vdash^\gamma_{\mathscr{C},\mathfrak{R}} \psi$. By construction of $\mathfrak{R}$, we know that for any such $\mathscr{C}$ there is a $v\in W'$ s.t. $Rwv$ and $\mathscr{C} = \mathscr{B}_v$ and that for all such $v$ there is a corresponding $\mathscr{C}$. We conclude by noting that $M',w\vDash^\gamma \square \psi$ iff for all $v$ s.t. $Rwv$, $M',v\vDash^\gamma \psi$ and the induction hypothesis. 

Given a reflexive $\gamma$, it no longer holds that for all $\mathscr{C}$ s.t. $\mathfrak{R}\mathscr{B}_w\mathscr{C}$ there is a $v\in W'$, $\mathscr{C}=\mathscr{B}_v$ and $Rwv$, as they could also be those subsets of a $\mathscr{B}_v$ added by (5c). However, by induction hypothesis, we have $\Vdash^\gamma_{\mathscr{B}_v,\mathfrak{R}} \psi$ for that $\mathscr{B}_v$ and we can conclude by noting that, for all $\mathscr{C}$ and $\mathscr{B}_v$ s.t. $\mathscr{B}_v$ is the only maximally-consistent base that is a superset of $\mathscr{C}$, $\Vdash^\gamma_{\mathscr{C},\mathfrak{R}} \psi$ iff $\Vdash^\gamma_{\mathscr{B}_v,\mathfrak{R}} \psi$, by Lemma \ref{BelowMaxCon}.

We conclude with step 6. From $M',w\nvDash^\gamma \phi$ and step 5, it follows that $\nVdash^\gamma_{\mathscr{B}_w, \mathfrak{R}} \phi$. Since, by step 4, $\mathfrak{R}$ is a $\gamma$-modal relation, $\phi$ is not valid in the base-extension semantics for $\gamma$.

We conclude that (for arbitrary $\phi$), if $\phi$ is not valid in the Kripke semantics for $\gamma$, then it is not valid in the base-extension semantics for $\gamma$ either.
\qed

\end{proof}

With this, we have now shown the soundness and completeness of our base-extension semantics and established Theorem \ref{Completeness}. These are, however, not the only proof strategies we could have used. We conjecture that the inverse to our proofs also hold (i.e., Validity in B-eS implies Theorem and Validity in Kripke semantics implies Validity in B-eS). While these results are unnecessary for our argument, because, as mentioned above, the proofs above suffice to establish soundness and completeness, they would still be interesting. On the one hand, proofs of both soundness and completeness between our base-extension semantics and our proof system shows that our base-extension semantics is not dependent on Kripke semantics. On the other hand, soundness and completeness proofs between the base-extension semantics and Kripke semantics highlight the strong relation between the two semantic approaches and give us an opportunity to analyse their connection. This is, however, left as a topic for further research.

\section{Duality}
\label{sec:Duality}

So far we have taken $\square$ as our primary operator and $\lozenge$ defined as its dual. 
We can, however, also define $\lozenge$ independently of $\square$ in the usual manner.  
\begin{definition}
    For the remainder of this section let the validity of $\lozenge$ be defined in the following way:
    \begin{longtable}{lcl}
         $\Vdash^\gamma_{\mathscr{B},\mathfrak{R}} \lozenge \phi$ & iff & for all $\mathscr{C}\supseteq \mathscr{B}$ there is a $\mathscr{C'}$ s.t. $\mathfrak{R}\mathscr{C}\mathscr{C'}$, $\Vdash^\gamma_{\mathscr{C'},\mathfrak{R}} \phi$\\
    
    \end{longtable}
\end{definition}

With this definition, $\lozenge$ and $\square$ are dual in the usual way.

\begin{lemma}
    \label{lem:Duality}

    The following hold for all $\gamma$, $\gamma$-modal relations $\mathfrak{R}$, and bases $\mathscr{B}$: 
    
    \noindent 1. $\Vdash^\gamma_{\mathscr{B},\mathfrak{R}}\lozenge\phi$ iff $\Vdash^\gamma_{\mathscr{B},\mathfrak{R}}(\square(\phi\to\bot)\to\bot)$ \;
        2. $\Vdash^\gamma_{\mathscr{B},\mathfrak{R}}\square\phi$ iff $\Vdash^\gamma_{\mathscr{B},\mathfrak{R}}(\lozenge(\phi\to\bot)\to\bot)$.     
\end{lemma}

    \begin{proof}
    Here is the layout for the proof of the first formula: 
\[
\begin{array}{l@{\quad}c@{\quad}l}
\Vdash^\gamma_{\mathscr{B},\mathfrak{R}} \lozenge\phi & \mbox{iff} & \mbox{(1) for all $\mathscr{C}\supseteq \mathscr{B}$, there is a $\mathscr{D}$ s.t. $\mathfrak{R}\mathscr{C}\mathscr{D}$, $\Vdash^\gamma_{\mathscr{D},\mathfrak{R}} \phi$} \\

& \mbox{iff} & \mbox{(2) for all $\mathscr{C}\supseteq \mathscr{B}$, there is a $\mathscr{D}$ s.t. $\mathfrak{R}\mathscr{C}\mathscr{D}$ and,}\\ 
& & \mbox{for all $\mathscr{F}\supseteq\mathscr{D}$ if $\Vdash^\gamma_{\mathscr{F},\mathfrak{R}} \phi\to\bot$, then $\Vdash^\gamma_{\mathscr{F},\mathfrak{R}} \bot$}\\

 & \mbox{iff} & \mbox{(3) for all $\mathscr{C}\supseteq \mathscr{B}$, there is a $\mathscr{D}$ s.t. $\mathfrak{R}\mathscr{C}\mathscr{D}$ and,}   \\ 
 & & \mbox{if $\Vdash^\gamma_{\mathscr{D},\mathfrak{R}} \phi\to\bot$, then $\Vdash^\gamma_{\mathscr{D},\mathfrak{R}} \bot$}\\
 
 & \mbox{iff} & \mbox{(4) for all $\mathscr{C}\supseteq \mathscr{B}$, there is a $\mathscr{D}$ s.t. $\mathfrak{R}\mathscr{C}\mathscr{D}$ and,}  \\
 & & \mbox{if $\Vdash^\gamma_{\mathscr{D},\mathfrak{R}} \phi\to\bot$, then $\Vdash^\gamma_{\mathscr{C},\mathfrak{R}} \bot$}\\
 
& \mbox{iff} & \mbox{(5) for all $\mathscr{C}\supseteq \mathscr{B}$ there are $\mathscr{E}\supseteq\mathscr{C}$ and $\mathscr{D}$ s.t. $\mathfrak{R}\mathscr{E}\mathscr{D}$ and,}\\
& & \mbox{if $\Vdash^\gamma_{\mathscr{D},\mathfrak{R}} \phi\to\bot$, then $\Vdash^\gamma_{\mathscr{C},\mathfrak{R}} \bot$}\\

& \mbox{iff} & \mbox{(6) for all $\mathscr{C}\supseteq \mathscr{B}$, if for all $\mathscr{E}\supseteq\mathscr{C}$ and $\mathscr{D}$ s.t. $\mathfrak{R}\mathscr{E}\mathscr{D}$,}\\
& & \mbox{if $\Vdash^\gamma_{\mathscr{D},\mathfrak{R}} \phi\to\bot$, then $\Vdash^\gamma_{\mathscr{C},\mathfrak{R}} \bot$}\\

\Vdash^\gamma_{\mathscr{B},\mathfrak{R}} \neg\square\neg\phi & \mbox{iff} & \mbox{(7) for all $\mathscr{C}\supseteq \mathscr{B}$, if $\Vdash^\gamma_{\mathscr{C},\mathfrak{R}} \square (\phi\to\bot)$, then $\Vdash^\gamma_{\mathscr{C},\mathfrak{R}} \bot$}\\

\end{array} 
\]
    
    The steps between (1) and (2) follow from double negation proved in Lemma \ref{DoubleNegation}. The interesting step in this proof is from (3) to (2): if $\Vdash^\gamma_{\mathscr{D},\mathfrak{R}} \phi\to\bot$, then $\Vdash^\gamma_{\mathscr{D},\mathfrak{R}} \bot$ implies that, for all $\mathscr{F}\supseteq\mathscr{D}$, if $\Vdash^\gamma_{\mathscr{F},\mathfrak{R}} \phi\to\bot$, then $\Vdash^\gamma_{\mathscr{F},\mathfrak{R}} \bot$ because for all $\mathscr{G}\supseteq\mathscr{F}$ s.t. $\Vdash^\gamma_{\mathscr{G},\mathfrak{R}} \phi\to\bot$, we also have $\mathscr{G}\supseteq\mathscr{D}$ and so $\Vdash^\gamma_{\mathscr{G},\mathfrak{R}} \bot$. The converse holds simply because $\mathscr{D}\supseteq\mathscr{D}$. The steps between (3) and (4) follow from $(a)$ of Definition \ref{modalRelation}. As in the step from (3) to (2), the step from (4) to (5) follows simply because $\mathscr{C}\supseteq\mathscr{C}$. The converse step follows from $(d)$ of Definition \ref{modalRelation}. Finally, the steps between (5) and (6) are just moving quantifiers in or outside the scope of the conditional and, similarly the steps between (6) and (7) hold by simple unpacking/packing of the validity conditions.
    
    The proof of the second formula follows the same pattern.
\qed

    \end{proof}

\section{Incompleteness for Euclidean Modal Logics}
\label{sec:Euclidean}

We have shown soundness and completeness for the modal logics $K$, $KT$, $K4$, and $S4$, but we have yet to discuss any modal logic that is euclidean. Because we have a base-extension semantics for $S4$, perhaps the next natural step would be to consider $S5$. 

However, with our specific current approach we will not be able to get to $S5$, as it is incomplete for any euclidean modal logic.


\begin{lemma}
    \label{lem:Euclidean}
    It is not the case that if $\gamma$ is euclidean, then (5) $\lozenge \phi\Vdash^\gamma \square\lozenge\phi$.
\end{lemma}

\begin{proof}

We show this by giving some $\mathscr{B}$ and euclidean $\mathfrak{R}$ s.t. $\Vdash_{\mathscr{B},\mathfrak{R}}^\gamma \lozenge\phi$, but $\nVdash_{\mathscr{B},\mathfrak{R}}^\gamma \square\lozenge\phi$. For simplicity, let $\mathscr{B}$ be maximally-consistent, so we do not have to deal with supersets of $\mathscr{B}$ and let $\phi$ be some propositional sentence $p$. Note that we only require $\mathfrak{R}$ to be euclidean and it does not need to be reflexive or transitive. We start by giving a euclidean relation $\mathfrak{R}^*$ that gives us our result and then extend it to be a modal-relation at the end.

Take the base $\mathscr{C} = \{\Rightarrow p\}$ and let $\mathscr{C}$ be the only base s.t. $\mathfrak{R}^*\mathscr{B}\mathscr{C}$.  By Definition \ref{EXTValidity},   $\Vdash_{\mathscr{C},\mathfrak{R}^*}^\gamma p$. This suffices for $\Vdash_{\mathscr{B},\mathfrak{R}^*}^\gamma \lozenge p$. Since we require $\mathfrak{R}^*$ to be euclidean, we also require $\mathfrak{R}^*\mathscr{C}\mathscr{C}$. Now take any propositional letter $q \neq p$ and base $\mathscr{D} = \{\Rightarrow q\}$ and let $\mathfrak{R}^*\mathscr{C}\mathscr{D}$. Again by Definition \ref{EXTValidity}, we have $\nVdash_{\mathscr{D},\mathfrak{R}^*}^\gamma p$. Finally, take any two bases $\mathscr{E}$ and $\mathscr{F}$ that are both maximally-consistent supersets of $\mathscr{C}$ but are not supersets of each other and let $\mathfrak{R}^*\mathscr{E}\mathscr{D}$ and $\mathfrak{R}^*\mathscr{F}\mathscr{C}$; that is, $\mathscr{E}$ inherits the relation to $\mathscr{D}$ and $\mathscr{F}$ the one to $\mathscr{C}$ from $\mathscr{C}$. Now since we do not have $\mathfrak{R}^*\mathscr{E}\mathscr{C}$, $\mathscr{C}$ does not have access to a base at which $\phi$ holds and so we get $\nVdash_{\mathscr{C},\mathfrak{R}^*}^\gamma \lozenge p$ and finally $\nVdash_{\mathscr{B},\mathfrak{R}^*}^\gamma \square\lozenge p$. 

As already mentioned $\mathfrak{R}^*$ is not a modal-relation. Consider the relation $\mathfrak{R}$ given only by the following:

\begin{enumerate}[label=(\arabic{enumi})]

\item if $\mathfrak{R}^*\mathscr{B}\mathscr{C}$, then $\mathfrak{R}\mathscr{B}\mathscr{C}$

\item if $\mathscr{B}$ and $\mathscr{C}$ are inconsistent, then $\mathfrak{R}\mathscr{B}\mathscr{C}$

\item if there is a consistent base $\mathscr{D}$ s.t. $\mathscr{B}\subseteq\mathscr{D}$ and $\mathfrak{R}\mathscr{D}\mathscr{C}$, then $\mathfrak{R}\mathscr{B}\mathscr{C}$.

\end{enumerate}

As in the proof of Theorem \ref{Soundness}, (2) guarantees (a) and (3) guarantees (d) of Definition \ref{modalRelation}. Condition (c) holds because the only relations given by (1) that do not originate from a maximally-consistent bases are $\mathfrak{R}\mathscr{C}\mathscr{C}$ and $\mathfrak{R}\mathscr{C}\mathscr{D}$, but, by construction of $\mathfrak{R}^*$ and (1), we have $\mathfrak{R}\mathscr{E}\mathscr{D}$ and $\mathfrak{R}\mathscr{F}\mathscr{C}$ and $\mathscr{E}$ and $\mathscr{F}$ are supersets of $\mathscr{C}$. For (b) it suffices that none of the steps allow a relation from a consistent to an inconsistent base.

We have $\Vdash_{\mathscr{B},\mathfrak{R}}^\gamma \lozenge\phi$, because, as before, $\mathfrak{R}\mathscr{B}\mathscr{C}$. We also have $\nVdash_{\mathscr{C},\mathfrak{R}}^\gamma \lozenge p$ and, unchanged by the construction of $\mathfrak{R}$, $\mathscr{C}$ is the only base s.t. $\mathfrak{R}\mathscr{F}\mathscr{C}$. So, $\nVdash_{\mathscr{C},\mathfrak{R}}^\gamma \lozenge p$ and, finally, $\nVdash_{\mathscr{B},\mathfrak{R}}^\gamma \square\lozenge p$
\qed
\end{proof}

This result shows that our conditions for a modal relation of Definition \ref{modalRelation} do not sufficiently preserve the structure of the relation when going from bases to their sub- or super-sets. It is left for future work to try to analyse and adapt conditions (c) and (d), specifically. We believe it is possible that, with some fine tuning, a definition of modal relation can be found that results in sound and complete semantics for all the modal logics discussed in this paper.


\section{Conclusion}
\label{sec:Conclusion}

We have developed base-extension semantics for the classical propositional modal systems $K$, $KT$, $K4$, and $S4$. We have established appropriate soundness and completeness theorems. We 
 have shown duality between $\square$ and a natural definition of $\lozenge$. We have also shown that our approach, at least as it is represented here, does not yield a complete semantics for euclidean modal logics. 


Future work might include the following: proof-theoretic semantics for euclidean modal logics (especially $S5$) and, based on that, proof-theoretic semantics for applied modal logics, such as DEL \cite{Baltag1998}, belief revision \cite{Alchourron1985}, and temporal systems \cite{Goldblatt1992}; proof-theoretic semantics for modal logics in terms of natural deduction systems (for example, consider a comparison between our treatment of $S4$ and Sandqvist's treatment of intuitionistic propositional logic in \cite{Sandqvist2015}, given that there are translations between these two logics (see \cite{gore2019})); 
proof-theoretic semantics for intuitionistic modal propositional logics; 
 and correspondence results, along the lines of Sahlqvist's theorem \cite{Sahlqvist1975}.

\nocite{*}

	\bibliographystyle{plain}
	\bibliography{modalPTS}

\appendix
\section{Appendix: Completeness wrt a classical Hilbert proof system}
\label{Appendix}
We present a different approach to proving that every classical tautology is valid in our base-extension semantics.

In Lemma \ref{Kmodalaxioms}, we have already shown that (MP) holds on our bases so what is left are the axioms (1)-(3) from Definition \ref{HilbertSystem}. For this we require 2 more lemmas. First we show that the law of excluded middle holds at maximally-consistent bases.

\begin{lemma}
    \label{lem:MaxConOr}
For every formula $\phi$, maximally-consistent base $\mathscr{B}$, and $\gamma$-modal relation $\mathfrak{R}$, either $\Vdash^\gamma_{\mathscr{B},\mathfrak{R}} \phi$ or $\Vdash^\gamma_{\mathscr{B},\mathfrak{R}} \phi\to\bot$.
\end{lemma}

\begin{proof}
    We show that if $\nVdash^\gamma_{\mathscr{B},\mathfrak{R}} \phi$, then $\Vdash^\gamma_{\mathscr{B},\mathfrak{R}} \phi\to\bot$. 
    
    By Lemma \ref{ModalBehaviour}, $\Vdash^\gamma_{\mathscr{B},\mathfrak{R}} \phi\to\bot$ iff if $\Vdash^\gamma_{\mathscr{B},\mathfrak{R}} \phi$, then $\Vdash^\gamma_{\mathscr{B},\mathfrak{R}} \bot$. So, since we assumed $\nVdash^\gamma_{\mathscr{B},\mathfrak{R}} \phi$, the antecedent is false and so $\Vdash^\gamma_{\mathscr{B},\mathfrak{R}} \phi\to\bot$.
    \qed
\end{proof}

The second lemma we require is a generalization of Lemma \ref{MaxCon} from classical to modal base-extension semantics.

\begin{lemma}
\label{lem: ModalMaxCon}
For every formula $\phi$, if there is a base $\mathscr{B}$ s.t. $\nVdash^\gamma_{\mathscr{B},\mathfrak{R}} \phi$, then there is a maximally-consistent base $\mathscr{B}^*\supseteq\mathscr{B}$ with $\nVdash^\gamma_{\mathscr{B}^*, \mathfrak{R}} \phi$.
\end{lemma}

\begin{proof}
    The propositional cases are the same as in the proof of Lemma \ref{MaxCon}. So, there is only one new case to consider: $\phi=\square\psi$. Since $\nVdash^\gamma_{\mathscr{B},\mathfrak{R}} \square\psi$, there are $\mathscr{C}\supseteq\mathscr{B}$,$\mathscr{D}$ s.t. $\mathfrak{R}\mathscr{C}\mathscr{D}$, and $\nVdash^\gamma_{\mathscr{D},\mathfrak{R}} \psi$. If $\mathscr{C}$ is maximally-consistent, we are done since $\nVdash^\gamma_{\mathscr{C},\mathfrak{R}} \phi$. Otherwise by (c), we know there is a $\mathscr{E}\supset\mathscr{C}$ with $\mathfrak{R}\mathscr{E}\mathscr{D}$. By repeating this step as much as necessary we eventually reach a maximally-consistent $\mathscr{B}^*\supseteq\mathscr{B}$ with $\mathfrak{R}\mathscr{B}^*\mathscr{D}$ and so $\nVdash_{\mathscr{B}^*,\mathfrak{R}} \square\psi$
    \qed
\end{proof}

Given these lemmas and the (MP) rule we have shown in Lemma \ref{Kmodalaxioms}, we can now show the axioms we require.

\begin{lemma}
    \label{lem:modalaxioms}
    For any $\gamma$, the following hold:
    \begin{enumerate}[label=(\arabic{enumi})]
         
        \item $\phi\Vdash^\gamma \psi\to\phi$ 
        \item $\phi\to(\psi\to\chi)\Vdash^\gamma (\phi\to\psi)\to(\phi\to\chi)$
        \item $(\phi\to\bot)\to(\psi\to\bot)\Vdash^\gamma \psi\to\phi$.
    \end{enumerate}
\end{lemma}

\begin{proof}
    By Definition \ref{EXTValidity}, (1) holds if for all $\mathscr{B}$ and $\mathfrak{R}$ s.t. $\Vdash^\gamma_{\mathscr{B},\mathfrak{R}}\phi$, we also have $\Vdash^\gamma_{\mathscr{B},\mathfrak{R}}\psi\to\phi$. For all $\mathscr{C}\supseteq\mathscr{B}$ s.t. $\Vdash^\gamma_{\mathscr{C},\mathfrak{R}}\psi$, we also have $\Vdash^\gamma_{\mathscr{C},\mathfrak{R}}\phi$ by Lemma \ref{ModalMonotonicity} and so we can conclude $\Vdash^\gamma_{\mathscr{B},\mathfrak{R}}\psi\to\phi$. 

    Following the same strategy for (2), we take $\Vdash^\gamma_{\mathscr{B},\mathfrak{R}}\phi\to(\psi\to\chi)$ and show that $\Vdash^\gamma_{\mathscr{B},\mathfrak{R}}(\phi\to\psi)\to(\phi\to\chi)$. For that we need for all $\mathscr{C}\supseteq\mathscr{B}$ s.t. $\Vdash^\gamma_{\mathscr{C},\mathfrak{R}}\phi\to\psi$ to also have  $\Vdash^\gamma_{\mathscr{C},\mathfrak{R}}\phi\to\chi$. To show that take an abitrary $\mathscr{D}\supseteq\mathscr{C}$ s.t. $\Vdash^\gamma_{\mathscr{D},\mathfrak{R}}\phi$. By Lemma \ref{ModalMonotonicity}, we have $\Vdash^\gamma_{\mathscr{D},\mathfrak{R}}\phi\to(\psi\to\chi)$ and $\Vdash^\gamma_{\mathscr{D},\mathfrak{R}}\phi\to\psi$. A couple of straightforward applications of (MP) get us to $\Vdash^\gamma_{\mathscr{D},\mathfrak{R}}\chi$ and so we can conclude $\Vdash^\gamma_{\mathscr{C},\mathfrak{R}}\phi\to\chi$ and, finally, $\Vdash^\gamma_{\mathscr{B},\mathfrak{R}}(\phi\to\psi)\to(\phi\to\chi)$.

    For (3) we, again, take $\Vdash^\gamma_{\mathscr{B},\mathfrak{R}}(\phi\to\bot)\to(\psi\to\bot)$. To show $\Vdash^\gamma_{\mathscr{B},\mathfrak{R}}\psi\to\phi$, we show that for all $\mathscr{C}\supseteq\mathscr{B}$ s.t. $\Vdash^\gamma_{\mathscr{C},\mathfrak{R}}\psi$ we also have $\Vdash^\gamma_{\mathscr{C},\mathfrak{R}}\phi$. So assume $\nVdash^\gamma_{\mathscr{C},\mathfrak{R}} \phi$. By Lemma \ref{lem:MaxConOr}, we know that for all maximally-consistent $\mathscr{C}^*\supseteq\mathscr{C}$ either $\Vdash^\gamma_{\mathscr{C}^*,\mathfrak{R}} \phi$ or $\Vdash^\gamma_{\mathscr{C}^*,\mathfrak{R}} \phi\to\bot$. By Lemma \ref{ModalMonotonicity}, we also know that $\Vdash^\gamma_{\mathscr{C}^*,\mathfrak{R}} (\phi\to\bot)\to(\psi\to\bot)$ and $\Vdash^\gamma_{\mathscr{C}^*,\mathfrak{R}} \psi$. So, if $\Vdash^\gamma_{\mathscr{C}^*,\mathfrak{R}} \phi\to \bot$, then by (MP) $\Vdash^\gamma_{\mathscr{C}^*,\mathfrak{R}} \psi\to\bot$ and $\Vdash^\gamma_{\mathscr{C}^*,\mathfrak{R}} \bot$, which is a contradiction as $\mathscr{C}^*$ is consistent. So, since $\phi$ holds at all maximally-consistent superset bases of $\mathscr{C}$, it follows, from the contrapositive of Lemma \ref{lem: ModalMaxCon}, that $\Vdash^\gamma_{\mathscr{C},\mathfrak{R}} \phi$. 
    \qed
\end{proof}

The following follows immediately: 

\begin{lemma}
    For $\gamma = K, KT, K4$, or $S4$, if $\phi$ is a classical tautology, then $\phi$ is valid in the base-extension semantics for $\gamma$. 
\end{lemma}

\end{document}